%% file: paper.tex
\documentclass[a4paper]{article}

\usepackage{amsmath,amsthm}
\usepackage{algorithmic,algorithm}
\usepackage{booktabs}
\usepackage{tikz}

\theoremstyle{plain}
\newtheorem{theorem}{Theorem}[section]
\newtheorem{lemma}[theorem]{Lemma}

\theoremstyle{remark}
\newtheorem{definition}[theorem]{Definition}
\newtheorem{remark}[theorem]{Remark}

\DeclareMathOperator{\id}{id}

\DeclareMathOperator{\Sign}{Sign}
\DeclareMathOperator{\supp}{supp}
\DeclareMathOperator*{\argmin}{argmin}
\newcommand{\seq}[1]{(#1)}
\newcommand{\norm}[2][]{\|{#2}\|_{#1}}
\newcommand{\scp}[3][]{\langle{#2},\, {#3}\rangle_{#1}}
\newcommand{\NN}{\mathbf{N}}
\newcommand{\bigO}{\mathcal{O}}

\title{A projection proximal-point algorithm for
  $\ell^1$-minimization}

\author{Dirk A. Lorenz\thanks{Institute for Analysis and Algebra, TU
    Braunschweig, D-38092 Braunschweig, Germany,
    \texttt{d.lorenz@tu-braunschweig.de}}.}

\begin{document}
\maketitle
\begin{abstract}
  The problem of the minimization of least squares functionals with
  $\ell^1$ penalties is considered in an infinite dimensional Hilbert
  space setting. While there are several algorithms available in the
  finite dimensional setting there are only a few of them which come
  with a proper convergence analysis in the infinite dimensional
  setting.
  
  In this work we provide an algorithm from a class which have not
  been considered for $\ell^1$ minimization before, namely a
  proximal-point method in combination with a projection step.  We
  show that this idea gives a simple and easy to implement algorithm.
  We present experiments which indicate that the algorithm may perform
  better than other algorithms if we employ them without any special
  tricks. Hence, we may conclude that the projection proximal-point
  idea is a promising idea in the context of $\ell^1$-minimization.
\end{abstract}

\section{Introduction}
\label{sec:introduction}
In this work, we consider the $\ell^1$-minimization optimization problem. Let $K:
\ell^2 \to H$ be a bounded linear operator mapping the sequence space
$\ell^2$ into a Hilbert space $H$, $g \in H$ and $\alpha>0$. The
minimization problem reads as
\begin{equation}
  \label{eq:orig_min_prob}
  \min_{u\in\ell^2}\Psi(u)\quad\text{with}\quad \Psi(u)= \tfrac12\norm[H]{Ku - g}^2 + \alpha\norm[1]{u}.
\end{equation}
We follow~\cite{solodov1999hybridproximalpoint} and
derive a projection proximal-point algorithm which sequentially solves
a regularized problem
\begin{equation}
  \label{eq:reg_min_prob}
  \min_{u\in\ell^2} \Psi(u) +\tfrac{\mu_n}{2}\norm[2]{u - u^n}^2
\end{equation}
up to desired accuracy and then applies a projection which reduces the
distance to the minimizer of the original
problem~(\ref{eq:orig_min_prob}).  While the regularized
problem~(\ref{eq:reg_min_prob}) is still non-smooth, it turns out that
it can be solved easier than the original
one~(\ref{eq:orig_min_prob}).

The main aim of this article is, to provide another alternative
approach to $\ell^1$-minimization in the infinite dimensional setting.
Other approaches use surrogate
functionals~\cite{daubechies2003iteratethresh}, proximal
forward-backward splitting~\cite{combettes2005signalrecovery} or
generalized gradient
methods~\cite{bredies2005gencondgrad,bredies2008itersoftconvlinear}.
While all the mentioned approaches lead to the same iterated
soft-thresholding procedure, other methods use iterated
hard-thresholding~\cite{bredies2008harditer} or an active set approach
formulated as a semismooth Newton's
method~\cite{griesse2008ssnsparsity}.  We remark that in finite
dimensions several other algorithms are available as, e.g., gradient
type methods like GPSR~\cite{figueiredo2007gradproj} or fixed point
continuation~\cite{hale2008fixedpointcontinuation}, interior point
methods~\cite{kim2007l1ls} or active set methods like LARS and the
homotopy
method~\cite{efron2004lars,osborne2000variableselection,donoho2008fastl1}
to name just a few.  The contribution of the paper is hence twofold:
one the one hand, we add another class of algorithms, namely a
proximal-point-like algorithm, to the zoo of available methods and on
the other hand we provide an algorithm which is globally and strongly
convergent in the infinite dimensional setting. We stress, that the
purpose of this paper is not to develop an algorithm which outperforms
other existing methods.

The article is organized as follows: In
Section~\ref{sec:proj-prox-point} we review the projection
proximal-point algorithm for the solution of maximally monotone
operator equations. In Section~\ref{sec:iter-thresh-algor} we show how
the subproblems in the $\ell^1$ case can be solved by iterative
thresholding or the generalized conditional gradient method. In
Section~\ref{sec:full-algor-conv} we state the full algorithm and show
linear convergence of the method. Section~\ref{sec:numer-exper}
presents numerical experiments and Section~\ref{sec:conclusion}
concludes the article.

\section{The projection proximal-point algorithm}
\label{sec:proj-prox-point}
In this section we review briefly the projection proximal-point
algorithm from~\cite{solodov1999hybridproximalpoint}.  Further we show
how it can be applied in the context of $\ell^1$-minimization.

\subsection{The projection proximal-point algorithm for general maximal monotone inclusions}
\label{sec:proj-prox-point-1}
The projection proximal-point algorithm has been proposed to solve the
following inclusion problem: Let $H$ be Hilbert space and $T$ a
maximal monotone operator on $H$. Find $u\in H$ such that
\begin{equation}
  \label{eq:inclusion_T}
  0\in T(u).
\end{equation}
The algorithm iteratively solves a regularized subproblem: For a
given iterate $u^n$ and a parameter $\mu_n>0$ find $u^{n+1}$ as an
approximate solution of
\begin{equation}
  \label{eq:inclusion_T_regularized}
  0\in T(u) + \mu_n (u-u^n)
\end{equation}
The notion of ``approximate solution'' is made precise by saying that
$u^{n+1}$ is an approximate solution
of~(\ref{eq:inclusion_T_regularized}) if for small $\epsilon^n$ it
holds
\begin{equation}
  \label{eq:inclusion_T_regularized_approximate}
  0 = v^{n+1} + \mu_n (u^{n+1}-u^n) + \epsilon^n,\quad v^{n+1}\in T(u^{n+1}).
\end{equation}
After the solution of this approximate problem, a projection step
follows which provably reduces the distance to the solution set
of~(\ref{eq:inclusion_T}).  In general the regularized
subproblem~(\ref{eq:inclusion_T_regularized}) may be as hard as the
original problem~(\ref{eq:inclusion_T}) but it turns out that often it
is comparably easy to solve.  The total algorithm is stated as
Algorithm~\ref{alg:ppp_general}.
\begin{algorithm}
  \caption{Projection proximal-point algorithm}
  \label{alg:ppp_general}
  \begin{algorithmic}[1]
    \REQUIRE $u^0\in H$, $\sigma\in [0,1[$, $n_\text{max}$
    \STATE Set $n\leftarrow0$ and done $\leftarrow$ false.
    \WHILE {$n < n_\text{max}$ and not done}
    \STATE Choose $\mu_n>0$.
    \STATE
    Calculate $v^n$, $y^n$ and $\epsilon^n$ such that
    \begin{align*}
      0 & = v^n + \mu_n(y^n-u^n) + \epsilon^n\\
      v^n & \in T(y^n)\\
      \norm{\epsilon^n} & \leq \sigma\max\{\norm{v^n},\mu_n\norm{y^n-u^n}\}.
    \end{align*}
    \IF{$v^n=0$ or $y^n = u^n$}
    \STATE done $\leftarrow$ true
    \ELSE
    \STATE Calculate $u^n$ as
    \[
    u^{n+1} \leftarrow u^n - \frac{\scp{v^n}{u^n-y^n}}{\norm{v^n}^2}v^n.
    \]
    \STATE Set $n \leftarrow n+1$.
    \ENDIF
    \ENDWHILE
  \end{algorithmic}
\end{algorithm}

Step~4 of the algorithm solves the regularized subproblem
approximately up to a desired accuracy. The accuracy is tuned by the
parameter $\sigma$. The parameters $\mu_n$ give the amount of
regularization of the subproblem to be solved in step 4. Note that
step~8 is actually the projection onto the hyper-plane $S_n = \{u\in
H\ :\ \scp{v^n}{u-y^n}=0\}$. Note that by the condition on
$\norm{\epsilon^n}$ in step~4 one concludes that $S_n$ separates the
iterate $u^n$ from the solution set of~(\ref{eq:inclusion_T})
(see~\cite[Theorem 2.2]{solodov1999hybridproximalpoint}).

\subsection{The projection proximal-point algorithm for $\ell^1$-minimization}
\label{sec:proj-prox-point-2}
We apply Algorithm~\ref{alg:ppp_general} to the minimization
problem~(\ref{eq:orig_min_prob}).  The maximal monotone operator $T$
is given as the subgradient of the objective functional $\Psi$. With
the multivalued sign-function
\[
\Sign(u)_k \in 
\begin{cases}
  \{1\}, & u_k>0\\
  [-1,1], & u_k = 0\\
  \{-1\}, & u_k<0
\end{cases}.
\]
this problem reads as
\begin{equation*}
  0 \in K^*(Ku - g) + \alpha\Sign(u).
\end{equation*}
It is an easy observation that the regularized subproblem correspoding
to~(\ref{eq:inclusion_T_regularized}) is
\begin{equation}
  \label{eq:inclusion_ell1_regularized}
  0 \in K^*(Ku - g) + \alpha\Sign(u) + \mu_n(u - u^n).
\end{equation}
and is equivalent to the problem
\begin{equation*}
  \min_{u\in\ell^2} \tfrac{1}{2}\norm[H]{Ku-g}^2 + \alpha\norm[1]{u} + \tfrac{\mu_n}{2}\norm[2]{u - u^n}^2.
\end{equation*}
Hence, the subproblems are regularized by adding the $\ell^2$-distance
to the previous iterate. 

The crucial step in Algorithm~\ref{alg:ppp_general} is step~4 where
one needs to solve the regularized problem up to a desired accuracy.
In Section~\ref{sec:iter-thresh-algor} we describe two algorithms with
are shown to produce solutions with the desired accuracy iteratively.
Hence, our algorithms consists of two nested loops. At first glance
one may think that this will result in bad performance but it turns
out that the inner loop usually terminates quite fast (depending on
the choice of $\mu_n$ and $\sigma$). Before turning to the full
algorithm in the case of $\ell^1$-minimization we present two
algorithms to solve the subproblem in step 4 of
Algorithm~\ref{alg:ppp_general}.

\section{Iterative thresholding algorithms for the regularized subproblem}
\label{sec:iter-thresh-algor}
Every iteration of the projection proximal-point algorithm involves the approximate solution of a regularized problem of the form
\begin{equation}
  \label{eq:reg_min_prob_var}
  \min_u \tfrac12\norm{Ku-g}^2 + \alpha\norm[1]{u} + \tfrac{\mu_n}{2}\norm[2]{u-u^n}^2
\end{equation}
In the next subsection we show, that these subproblems can be solved by means of either the a generalized gradient projection method from \cite{bredies2008itersoftconvlinear} (which is in fact a damped iterative soft-thresholding) or a generalized conditional gradient method from~\cite{bredies2005gencondgrad,bredies2008harditer}.

\subsection{Damped iterative soft-thresholding}
\label{sec:damp-iter-soft}
In this subsection we assume that $\norm{K}\leq 1$ (a condition which
may alwas be fulfilled by rescaling the problem). To derive an
algorithm for the approximate solution of the regularized
problem~(\ref{eq:reg_min_prob_var}) we use the
characterization~(\ref{eq:inclusion_ell1_regularized}) for a solution
$\bar u$:
\begin{equation*}
  0 \in K^*(K\bar u - g) + \alpha \Sign(\bar u) + \mu_n(\bar u - u^n).
\end{equation*}
We rewrite the characterization as
\[
\bar u - K^*(K\bar u - g) + \mu_n u^n \in ((1+\mu_n)\id + \alpha\Sign)(\bar u)
\]
which leads to
\[
\tfrac{1}{1+\mu_n}(\bar u - K^*(K\bar u - g) + \mu_n u^n) \in (\id + \tfrac{\alpha}{1+\mu_n}\Sign)(\bar u).
\]
Since $\Sign$ is a maximal monotone operator (as the subgradient of a
proper, convex, and lower-semicontinuous functional), $(\id + c\Sign)$
possesses a single valued inverse for any $c>0$.  This can be given
explicitly as
\[
S_c(u) := (\id + c\Sign)^{-1}(u)_k = 
\begin{cases}
  u_k - c, & u_k > c\\
  0, & |u_k| \leq c\\
  u_k + c, & u_k < -c
\end{cases}
\]
Hence, $\bar u$ fulfills the fixed-point equation
\begin{eqnarray*}
  \bar u & = S_{\tfrac{\alpha}{1+\mu_n}}(\tfrac{1}{1+\mu_n}(\bar u - K^*(K\bar u - g) + \mu_n u^n))\\
  & = \tfrac{1}{1+\mu_n}S_\alpha(\bar u - K^*(K\bar u - g) + \mu_n u^n).
\end{eqnarray*}
Since the operator $S_\alpha$ is non-expansive and we assumed that
$\norm{K}\leq 1$, we see that the mapping $u\mapsto
(1+\mu_n)^{-1}S_\alpha(u - K^*(Ku - g) + \mu_n u^n)$ is a contraction
with constant $(1+\mu_n)^{-1}$.  By the Banach fixed point theorem we
have the following result.
\begin{theorem}
  \label{thm:conv_damp_iter_soft}
  The iterates
  \begin{equation}
    \label{eq:soft_iter_subprob}
    y^{k+1} = \tfrac{1}{1+\mu_n}S_\alpha(y^n - K^*(Ky^k - g) + \mu_n u^n)
  \end{equation}
  converge with linear rate to the solution of~(\ref{eq:reg_min_prob_var}), especially it holds for the solution $\bar y$ of~(\ref{eq:reg_min_prob_var}) that
  \[
  \norm{y^{k+1} - \bar y} \leq \tfrac{1}{1+\mu_n}\norm{y^k-\bar y}.
  \]
\end{theorem}
\begin{remark}[Damped iterated soft-thresholding as generalized gradient projection method]
  \label{rem:iter_soft_thresh_as_ggpm}
  An alternative motivation for the above algorithm is as follows.
  We split the objective function as
  \[
  F(u) = \tfrac12\norm{Ku - g}^2,\qquad \Phi(u) = \alpha\norm[1]{u} + \tfrac{\mu_n}{2}\norm{u-u^n}^2.
  \]
  Now we apply the generalized gradient projection method from~\cite{bredies2008itersoftconvlinear} to the problem
  \[
  \min_u F(u) + \Phi(u).
  \]
  The algorithm is
  \[
  y^{k+1} = J_{s_k}(y^k - s_k F'(y^k))
  \]
  where $F'(y^k) = K^*(Ky^k - g)$ and $J_s$ is the proximal mapping
  \[
  J_s(y) = \argmin_v \tfrac12\norm{v-y}^2 + s\Phi(v).
  \]
  One easily verifies that
  \[
  J_s(y) = \tfrac{1}{1+s\mu_n}S_{s\alpha}(y + s\mu_n u^n)
  \]
  and hence the generalized gradient projection method gives
  \[
  y^{k+1} = \tfrac{1}{1+s_k\mu_n}S_{s_k\alpha}(y^k -s_k K^*(Ky^k - g)  + \mu_n u^n)
  \]
  which is the same as~(\ref{eq:soft_iter_subprob}) for $s_k=1$.
  In~\cite{bredies2008itersoftconvlinear} it is shown that the
  generalized gradient projection method converges as soon as the
  stepsizes fulfill $0<\underbar{s} \leq s_k\leq \bar s < 2/\norm{K}^2$.
\end{remark}

Finally we state the damped iterative soft-thresholding as
Algorithm~\ref{alg:dampitersoft}.
\begin{algorithm}
  \caption{Damped iterative soft-thresholding}
  \label{alg:dampitersoft}
  \begin{algorithmic}[1]
    \REQUIRE $y^0\in H$,  $k_\text{max}$, $\underbar{s}$, $\bar s$
    \STATE Set $k\leftarrow0$
    \WHILE {$k < k_\text{max}$}
    \STATE Choose $0<\underbar{s} \leq s_k\leq \bar s < 2/\norm{K}^2$ and update
    \begin{equation*}
      y^{k+1} = \tfrac{1}{1+s_k\mu_n}S_{s_k\alpha}(y^k -s_k K^*(Ky^k - g)  + \mu_n u^n)
    \end{equation*}
    \STATE Set $n \leftarrow n+1$
    \ENDWHILE
  \end{algorithmic}
\end{algorithm}

\subsection{Generalized conditional gradient method}
\label{sec:gener-cond-grad}
The generalized conditional gradient method as proposed and analyzed in
\cite{bredies2008harditer,bredies2005gencondgrad,bonesky2007gencondgradnonlin}
offers another
possibility for the approximate solution
of~(\ref{eq:reg_min_prob_var}).  As in
Remark~\ref{rem:iter_soft_thresh_as_ggpm} we split the objective
function as
\[
F(u) = \tfrac12\norm{Ku - g}^2,\qquad \Phi(u) = \alpha\norm[1]{u} + \tfrac{\mu_n}{2}\norm{u-u^n}^2.
\]
Now we apply the generalized conditional gradient method to $\min F(u)
+ \Phi(u)$.  In the case $F(u) = \tfrac12\norm{Ku-g}^2$ and a general
convex $\Phi$ the algorithm is given in
Algorithm~\ref{alg:gencondgrad}.
\begin{algorithm}
  \caption{Generalized conditional gradient method}
  \label{alg:gencondgrad}
  \begin{algorithmic}[1]
    \REQUIRE $y^0\in H$,  $k_\text{max}$
    \STATE Set $k\leftarrow0$
    \WHILE {$k < k_\text{max}$}
    \STATE
    Calculate a search direction $w^k$ as
    \begin{equation}
      \label{eq:search_direction_gcgm}
      w^k\in\argmin_w \scp{K^*(Ky^k - g)}{w} + \Phi(w).
    \end{equation}
    \STATE 
    Calculate a step-size $s_k$ according to 
    \[
    s_k = \min\left\{1, \frac{\Phi(y^k) - \Phi(w^k) + \scp{Ky^k - g}{K(y^k-w^k)}}{\norm{K(y^k-w^k)}^2}\right\}.
    \]
    \STATE
    Update $y^{k+1} = y^k + s_k(w^k-y^k)$ and set $n \leftarrow n+1$
    \ENDWHILE
  \end{algorithmic}
\end{algorithm}

For our special choice of $\Phi$ the search direction $w^k$ in~(\ref{eq:search_direction_gcgm}) is given by the following lemma.
\begin{lemma}
  Let $\Phi(u)= \alpha\norm[1]{u} + \tfrac{\mu_n}{2}\norm{u-u^n}^2$. The solution of~(\ref{eq:search_direction_gcgm}) is given by
  \[
  w = \tfrac{1}{\mu_n}S_\alpha(\mu_n u^n - K^*(Ky^k-g)).
  \]
\end{lemma}
\begin{proof}
  A solution $w$ of~(\ref{eq:search_direction_gcgm}) is characterized by
  \[
  0 \in K^*(Ky^k - g) + \alpha\Sign(w) + \mu_n(w-u^n)
  \]
  which we rewrite as
  \[
  u^n - \tfrac{1}{\mu_n}K^*(Ky^k - g) \in (\id + \tfrac{\alpha}{\mu_n}\Sign)(w).
  \]
  Similar to the calculation in Section~\ref{sec:damp-iter-soft} this leads to
  \[
  w = \tfrac{1}{\mu_n} S_\alpha(\mu_n u^n - K^*(Ky^k - g)).
  \]
\end{proof}

Using techniques from~\cite{bredies2008harditer} we derive the following result:
\begin{theorem}
  Let $\Phi(u)= \alpha\norm[1]{u} + \tfrac{\mu_n}{2}\norm{u-u^n}^2$.
  Then the iterates $y^k$ produced by Algorithm~\ref{alg:gencondgrad}
  converge with a linear rate to the unique solution
  of~(\ref{eq:reg_min_prob_var}).
\end{theorem}
\begin{proof}
  We use Theorem~7 from~\cite{bredies2008harditer}.  Let $y^*$ denote
  the unique solution of~(\ref{eq:reg_min_prob_var}).
  In~\cite[Theorem~7]{bredies2008harditer} it is shown, that one gets
  linear convergence of the iterates of the generalized conditional
  gradient method as soon as an estimate
  \begin{equation}
    \label{eq:bregman_estimate}
    \scp{K^*(Ky^* - g)}{v - y^*} + \Phi(v) - \Phi(y^*)
    \geq c\norm{v-y^*}^2
  \end{equation}
  holds locally around $y^*$.
  Since $y^*$ is a solution of~(\ref{eq:reg_min_prob_var}) we conclude that
  \[
  -K^*(Ky^* - g) \in \partial\Phi(y^*).
  \]
  and hence, the right hand side of~(\ref{eq:bregman_estimate}) a
  Bregman distance with respect to $\Phi$. By standard argument from convex analysis we conclude that the subdifferential of $\Phi$ at $y^*$ is $\alpha\Sign(y^*) + \mu_n(y^*-u^n)$ and hence we have
  \[
  -K^*(Ky^*-g) = \alpha w + \mu_n(y^*-u^n)\ \text{ with }\ w\in\Sign(y^*).
  \]
  Now we estimate
  \begin{multline*}
    \scp{K^*(Ky^* - g)}{v - y^*} + \Phi(v) - \Phi(y^*)
    \\
    = \sum_i -(\alpha w_i + \mu_n(y^*_i - u^n_i))(v_i - y^*_i) + \alpha(|v_i| - |y^*_i|) + \tfrac{\mu_n}{2}((v_i-u^n_i)^2 - (y^*_i - u^n_i)^2)\\
    = \sum_i  \alpha\underbrace{\bigl(|v_i| - |y^*_i| - w_i(v_i - y^*_i)\bigr)}_{\geq 0}  + \tfrac{\mu_n}{2}\bigl((v_i-u^n_i)^2 - (y^*_i-u^n_i)^2 - 2(y^*_i - u^n_i)(v_i - y^*_i)\bigr)\\
    \geq \tfrac{\mu_n}{2} \sum_i (v_i-u^n_i)^2 - (y^*_i-u^n_i)^2 - 2(y^*_i - u^n_i)(v_i - y^*_i)\\
    = \tfrac{\mu_n}{2}\sum_i (v_i - y^*_i)^2 =
    \tfrac{\mu_n}{2}\norm{v-y^*}^2.
  \end{multline*}
  Hence, we proved~(\ref{eq:bregman_estimate}) with $c=\mu_n/2$ and
  the claim follows from~\cite[Theorem~7]{bredies2008harditer}
\end{proof}

\section{Full algorithm and convergence properties}
\label{sec:full-algor-conv}

In the previous section we derived two algorithms which solve the
problem~(\ref{eq:inclusion_ell1_regularized}) which is needed in
step~4 of Algorithm~\ref{alg:ppp_general}. For a given $u^n$ both
algorithms produce iterates $y^k$ which converge with a linear rate to
the solution of~(\ref{eq:inclusion_ell1_regularized}). Hence, it
remains to check when to stop the iteration to fulfill the condition
in step~4 of Algorithm~\ref{alg:ppp_general}, namely:
\begin{align}
  v^n & \in T(y^k) \label{eq:cond_v_in_T_y} \\
  0 & = v^n + \mu_n(y^k-u^n) + \epsilon^n \label{eq:cond_def_epsilon}\\
  \norm{\epsilon^n} & \leq \sigma\max\{\norm{v^n},\mu_n\norm{y^k-u^n}\}
  \label{eq:cond_epsilon_less}.
\end{align}
To do so, we proceed as follows: Given an iterate of the outer
iteration $u^n$ and an iterate of the corresponding inner iteration
$y^k$ we define the projection onto the set $\alpha\Sign(y^k)$ as
\[
P_{\alpha\Sign(y^k)}(u)_i = 
\begin{cases}
  \alpha & y^k_i>0 \text{ or } (y^k_i=0 \text{ and } u_i >\alpha)\\
  -\alpha & y^k_i<0 \text{ or } (y^k_i=0 \text{ and } u_i < - \alpha)\\
  u_i & y^k_i= 0 \text{ and }  |u_i|\leq \alpha.
\end{cases}
\]
Then we calculate
\[
\epsilon^k = (\id - P_{\alpha\Sign(y^k)})(-K^*(Ky^k-g) - \mu_n(y^k-u^n))
\]
and
\[
v^k = -\mu_n(y^k - u^n) - \epsilon^k.
\]
It is obvious that then~(\ref{eq:cond_def_epsilon}) is fulfilled.
Moreover one sees
\begin{align*}
  v^k & =  -\mu_n(y^k - u^n) - (\id - P_{\alpha\Sign(y^k)})(-K^*(Ky^k-g) - \mu_n(y^k-u^n))\\
  & = K^*(Ky^k - g) + P_{\alpha\Sign(y^k)}(-K^*(Ky^k-g) - \mu_n(y^k-u^n))\\
  & \in K^*(K y^k - g) + \alpha \Sign(y^k) = T(y^k).
\end{align*}
Finally, it remains to check the inequality~(\ref{eq:cond_epsilon_less})
to accept an iterate $y^k$.

Putting the pieces together, we get the projection proximal-point algorithm for $\ell^1$-minimization as Algorithm~\ref{alg:ppp_l1}.
\begin{algorithm}
  \caption{Projection proximal-point algorithm for $\ell^1$-minimization}
  \label{alg:ppp_l1}
  \begin{algorithmic}[1]
    \REQUIRE $u^0\in H$, $\sigma\in [0,1[$, $n_\text{max}$
    \STATE Set $n\leftarrow0$ and done$_\text{outer}$ $\leftarrow$ false.
    \WHILE {$n < n_\text{max}$ and not done$_\text{outer}$}
    \STATE Choose $\mu_n>0$, and $y^0$.
    \STATE Set $k\leftarrow 1$ and done$_\text{inner}$ $\leftarrow$ false.
    \WHILE {not done$_\text{inner}$}
    \STATE Calculate $y^{k}$ via one step of either Algorithm~\ref{alg:dampitersoft} or~~\ref{alg:gencondgrad}.
    \STATE Calculate
    \begin{align*}
      \epsilon^k &= (\id - P_{\alpha\Sign(y^k)})(-K^*(Ky^k-g) - \mu_n(y^k-u^n))\\
      v^k & = -\mu_n(y^k - u^n) - \epsilon^k.
    \end{align*}
    \IF{$\norm{\epsilon^k} \leq \sigma\max\{\norm{v^k},\mu_n\norm{y^k-u^n}\}$}
    \STATE done$_\text{inner}$ $\leftarrow$ true
    \ELSE
    \STATE Set $k \leftarrow k+1$.
    \ENDIF
    
    \ENDWHILE
    \IF{$v^n=0$ or $y^n = u^n$}
    \STATE done$_\text{outer}$ $\leftarrow$ true
    \ELSE
    \STATE Calculate $u^n$ as
    \[
    u^{n+1} \leftarrow u^n - \frac{\scp{v^n}{u^n-y^n}}{\norm{v^n}^2}v^n.
    \]
    \STATE Set $n \leftarrow n+1$.
    \ENDIF
    \ENDWHILE
  \end{algorithmic}
\end{algorithm}

%\subsection{Convergence of the outer loop}
%\label{sec:conv-outer-loop}

The projection proximal-point algorithm is known to converge
$Q$-linearly under certain requirements. We cite
from~\cite{solodov1999hybridproximalpoint}:
\begin{theorem}
  \label{thm:q-conv_ppp_general}
  Let $T$ be maximally monotone, let $\bar u$ be a solution
  of~(\ref{eq:inclusion_T}) and let $\mu_n$ be bounded from above. If
  moreover there exists a constant $L>0$ such that for all $y,v$ with
  $\norm{v}\leq\delta$ and such that $v\in T(y)$ it holds
  \[
  \norm{y-\bar u} \leq L\norm{v}
  \]
  then the projection proximal-point algorithm~\ref{alg:ppp_general}
  converges $Q$-linearly.
\end{theorem}
In the special case of $\ell^1$-minimization the above theorem is
applicable if the operator $K$ obeys the finite basis injectivity
property from~\cite{bredies2008itersoftconvlinear}:
\begin{definition}
  \label{def:fbi}
  An operator $K:\ell^2\to H$ mapping into a Hilbert space has the
  \emph{finite basis injectivity} (FBI) property, if for all finite subsets
  $I\subset\NN$ the operator $K|_I$ is injective, i.e.~for all $u,v\in
  \ell^2$ with $Ku = Kv$ and $u_k = v_k = 0$ for all $k \notin I$ it
  follows $u = v$.
\end{definition}

\begin{theorem}
  \label{thm:q-conv_ppp_l1}
  Let $K:\ell^2\to H$ obey the FBI property. Then
  Algorithm~\ref{alg:ppp_l1} for $\ell^1$-minimization converges
  globally and $Q$-linearly.
\end{theorem}
\begin{proof}
  We show that the conditions in Theorem~\ref{thm:q-conv_ppp_general}
  are fulfilled. Hence, we consider $T(u) = K^*(Ku-g) +
  \alpha\Sign(u)$ and
  $0\in T(\bar u)$ and $v\in T(y)$.
  
  First we show, that $v\to 0$ implies $y\to\bar u$.
  To this end, we remark that $y$ solves
  \[
  y\in\argmin_u \Psi(u) - \scp{u}{v}.
  \]
  We now show, that the functional $\Psi(u) - \scp{u}{v}$
  $\Gamma$-converges to $\Psi$.  Consider a sequence $\seq{v_j}$ in
  $\ell^2$ with $v_j\to 0$ and define
  \[
  \Psi_j(u) = \Psi(u) - \scp{v_j}{u}.
  \]
  Since $\Psi$ is lower semi-continuous it holds for every $u_j\to u$
  that
  \[
  \liminf_j \Psi_j(u_j) = \liminf_j \Psi(u_j) - \scp{u_j}{v_j}
  \geq \Psi(u).
  \]
  Moreover, for the constant sequence $u_j=u$ we see that
  \[
  \limsup_j \Psi_j(u_j) = \limsup_j \Psi(u) - \scp{u}{v_j}
  = \Psi(u)
  \]
  and hence, $\Psi_j$ $\Gamma$-converges to $\Psi$ (see,
  e.g.~\cite{braides2002gammaconvergence}) and in particular the
  minimizers of $\Psi_j$ converge to that of $\Psi$, i.e.~for $y_j$
  such that $v_j\in T(y_j)$ it holds that $y_j\to\bar u$.  In other
  words: For $\epsilon>0$ there exists $\delta>0$ such that
  $\norm{v}\leq\delta$ implies $\norm{y-\bar u}\leq \epsilon$.
  
  For some $s>0$ the quantities $\bar u,v, y$ are characterized by
  \begin{align*}
    \bar u = S_\alpha(\bar u - K^*(K\bar u -g))\quad\text{and}\quad y
    = S_\alpha(y - K^*(Ky - g) + v).
  \end{align*}
  Similar to \cite[Proposition 3.10]{griesse2008ssnsparsity} one sees,
  that for $\norm{y-\bar u}\leq\epsilon$ and $\norm{v}\leq\delta$
  small enough there exists a number $k_0$ such that
  \[
  \{k\in\NN\ |\ |y - s K^*(Ky - g) + v|_k\leq s\alpha\} \subset
  \{1,\dots,k_0\} =: I.
  \]
  This shows, that the supports $y$ are uniformly bounded in a
  neighborhood of $\bar u$.
  Moreover, for $0<s<2/\norm{K}^2$ it holds
  \begin{align*}
  \norm{\bar u - y} &
  = \norm{S_\alpha(\bar u - K^*(K\bar u -g)) - S_\alpha(y - K^*(Ky - g) + v)}\\
  & \leq \norm{(\id-K^*K)(\bar u - y) + v}\\
  & \leq \norm{(\id-K^*K)(\bar u - y)} + \norm{v}
  \end{align*}
  Since $K$ obeys the FBI property and $\supp y,\supp u\subset I$,
  there exists $c\in]0,1[$ such that $\norm{(\id-K^*K)(\bar u -y)}\leq
  c\norm{\bar u - y}$ and we finally conclude
  \[
  \norm{\bar u -y}\leq \frac{1}{1-c}\norm{v}.
  \]
  The $Q$-linear convergence now follows from
  Theorem~\ref{thm:q-conv_ppp_general}.
\end{proof}

%\subsection{Termination of the inner loop}
%\label{sec:term-inner-loop}

\section{Numerical experiments}
\label{sec:numer-exper}

%Example run, comparison to iter_tresh, effect of projection
In this section we present example calculations to illustrate how the
algorithm works in different settings.

\subsection{Digital holography}
\label{sec:digital-holography}
As a first example problem we consider the problem of digital
holography.  In digital holography, the data correspond to the
diffraction patterns of the
objects~\cite{goodman2005ifo,kreis2005hhi}.  Under Fresnel's
approximation, diffraction can be modeled by a convolution with a
``chirp'' kernel.  In the context of holograms of
particles~\cite{vikram1992pfh,hinsch1995tdp,hinsch2002hpi}, the
objects can be considered opaque (i.e., binary) and the hologram
recorded on the camera corresponds to the convolution of disks with
Fresnel's chirp kernels. The measurement of particle size and location
therefore amounts to an inverse
problem~\cite{soulez2007holography,soulez2007holography2,Denis2009}.

We consider the problem of locating small objects (points), i.e. we
assume that opaque objects are distributed in three-dimensional space
for which we use the coordinates $(x,y,z)$. For objects which are located in the plane $z_j$ and an icident laser beam
in $-z$-direction of wavelength $\lambda$, the amplitude in the
observation plane at $z=0$ is well modeled by a bidimensional
convolution with respect to the variables $(x,y)$ with the Fresnel function
\[
h_{z_j}(x,y) = \frac{1}{i \lambda z_j} \exp \Big(i \frac{\pi}{\lambda
  z_j} (x^2+y^2) \Big).
\]
However, one is only able to measure the absolute value and not the
complex valued amplitude. After simplification, the problem of
reconstruction of a diffraction pattern from objects in the $z_j$-plane
can be modelled as deconvolution with the following kernel, see
Figure~\ref{fig:conv_kernel_dh}
\[
\kappa_{z_j}(x,y) = \real{h_{z_j}(x,y)} =
\frac{\sin(\tfrac{\pi}{\lambda z_j}(x^2+y^2))}{\lambda z_j}.
\]
\begin{figure}[]
  \centering
  \includegraphics[width=3cm]{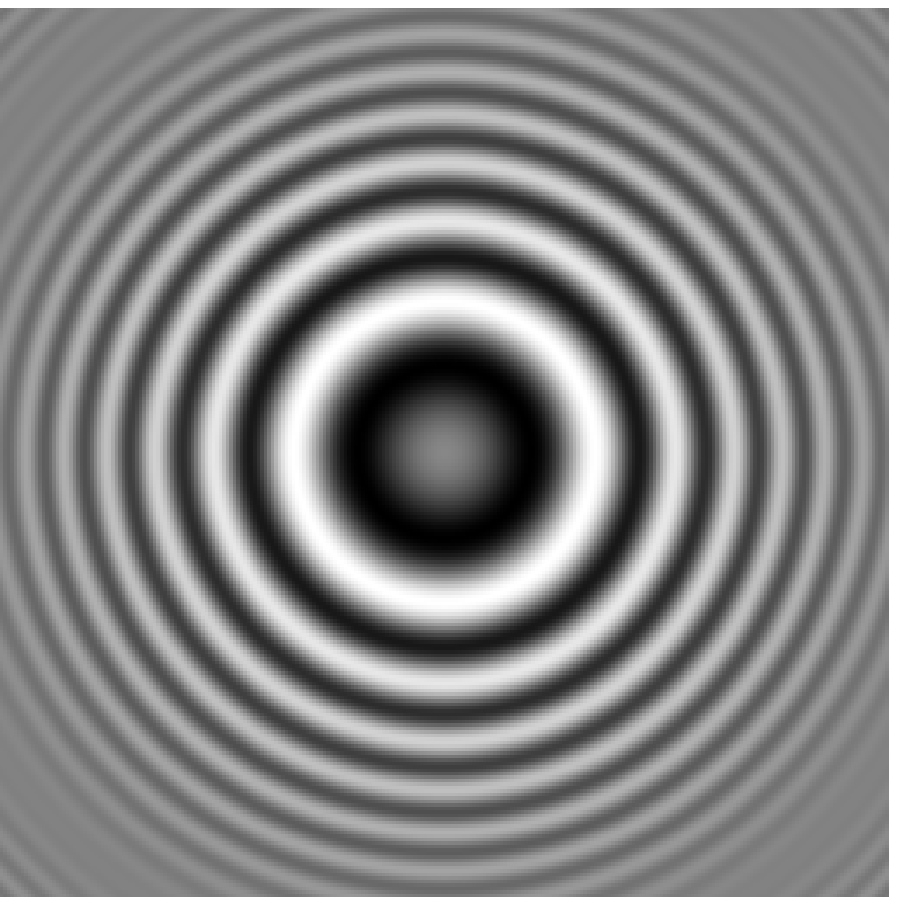}
  \includegraphics[width=3cm]{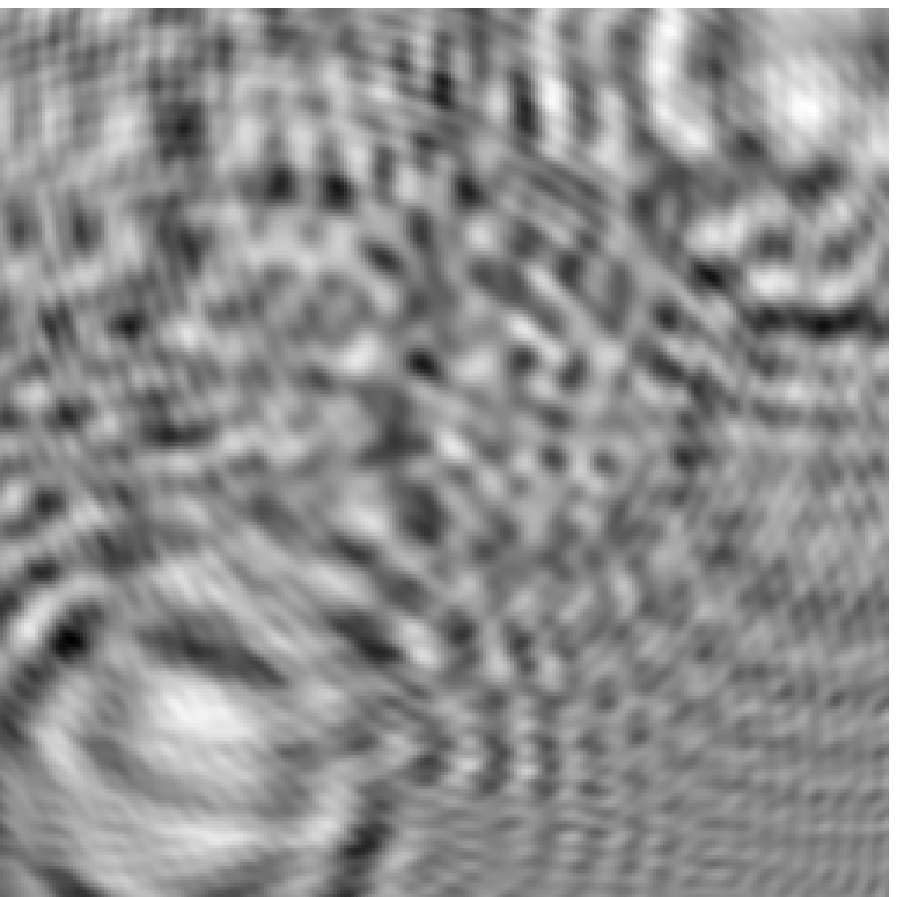}
  \caption{Left: Convolution kernel $\kappa$ for digital holography $\lambda=630$nm and recording distance of 250mm. Right: Simulated hologram with 20 particles.}
  \label{fig:conv_kernel_dh}
\end{figure}

We generated a hologram $g$ with a number of particles all located in
the same plane.\footnote{The author would like to thank Lo\"ic Denis
  (\'{E}cole Sup\'{e}rieure de Chimie Physique \'{E}lectronique de
  Lyon) for providing the implementation of the hologram simulator.}
Then we applied the projection proximal-point algorithm to the solution of
the inverse problem with the operator
\[
Ku = \kappa_{z_j}\ast u.
\]
We show reconstruction of the objects with different methods in
Figure~\ref{fig:rec_dh}. Note that all methods delivered well
seperated and moderately sharp objects. However, the result for the
projection proximal-point algorithm with generalized conditional
gradient method gives a slightly sharper reconstruction.
\begin{figure}
  \centering
  \small
  \begin{tabular}{cc}
  \includegraphics[width=3cm]{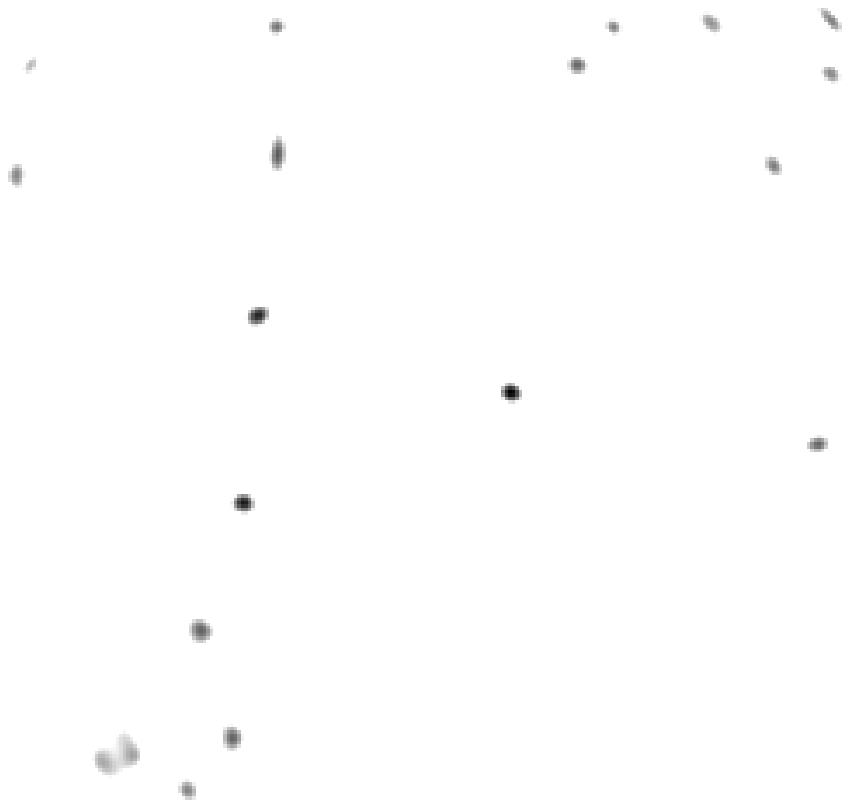}&
  \includegraphics[width=3cm]{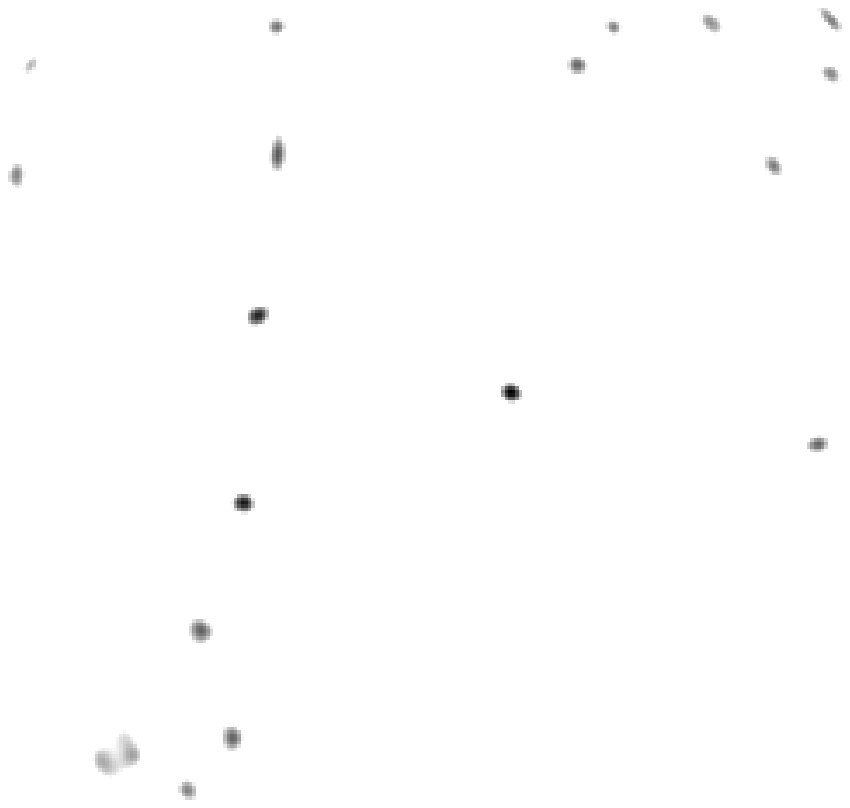}\\
  ppp, soft thresholding & soft thresholding\\
  \includegraphics[width=3cm]{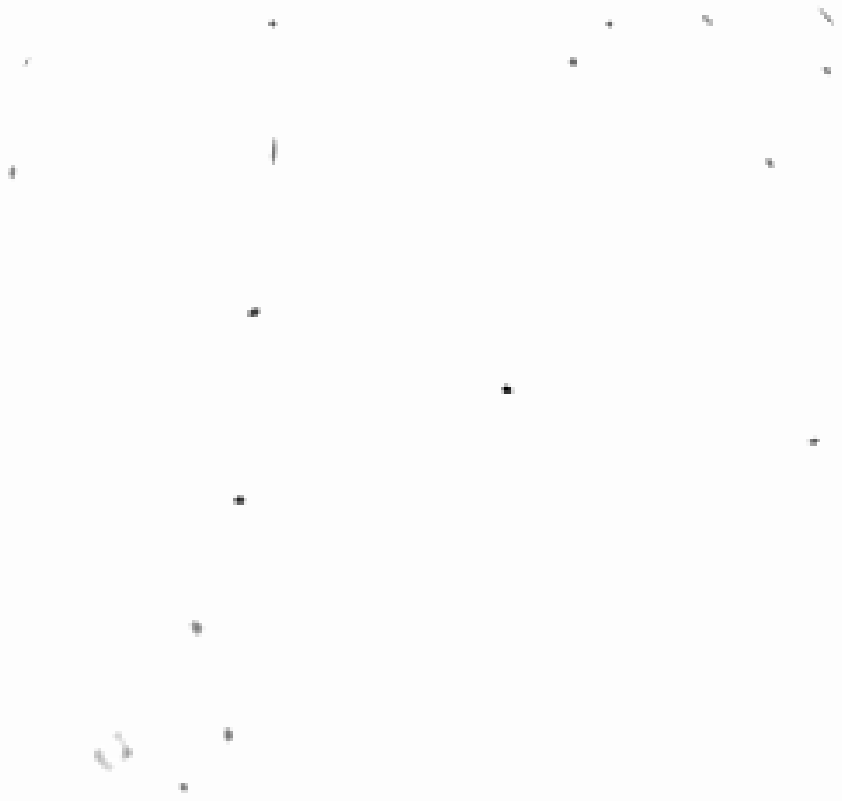}&
  \includegraphics[width=3cm]{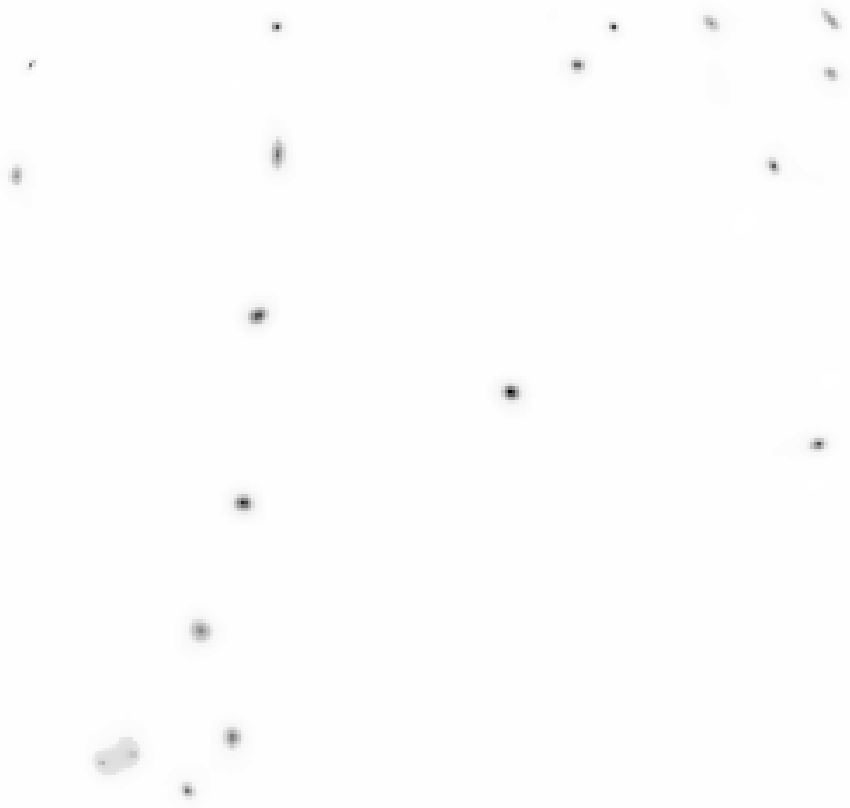}\\
  ppp, gen.~cond.~grad & hard thresholding
  \end{tabular}
  \caption{Reconstructions of the hologram from
    Figure~\ref{fig:conv_kernel_dh} with different methods and 120
    iterations. Top left: Projection proximal-point method with damped
    soft thresholding in the inner loop. Top right: Iterated soft
    thresholding~\cite{daubechies2003iteratethresh}. Bottom left:
    Projection proximal-point method with generalized conditional
    gradient method (Algorithm~\ref{alg:gencondgrad}) in the inner
    loop. Bottom right: Iterated hard
    thresholding~\cite{bredies2008harditer}.}
  \label{fig:rec_dh}
\end{figure}

In Figure~\ref{fig:func_val_dh} the development of the functional
value is shown. We remark, that the projection proximal-point
algorithm does not necessarily reduce the functional value and
sometimes, depending on the problem and the parameters, does not
produce monotonically decaying objective values. However, this may be
an advantage of the algorithms since it does not have to follow a
decent direction.  Note that iterated hard-thresholding (which is also
a generalized conditional gradient method) performs worst and that the
iterated soft-thresholding alone or in combination with the
proximal-point modification behave better and somewhat similar.
However, combining the generalized conditional gradient method with
the proximal-point modification gives significant improvement.

\begin{figure}
  \begin{tikzpicture}[xscale=.05,yscale=60]
    \footnotesize
    \draw[->] (0,0.42) -- (120,0.42) node[below] {$n$};
    \foreach \x in {100}
    \draw (\x,0.421) -- (\x,0.419) node[below] {\footnotesize$\x$};
    \draw[->] (0,0.42) -- (0,0.505) node[left] {$\Psi(u^n)$};
    \draw (2,.45) -- (-2,.45) node[left] {\footnotesize1.345};
    \draw (2,.5) -- (-2,.5) node[left] {\footnotesize1.350};
    \begin{scope}
      \clip (-20,.4) rectangle (180,.505);
      \draw[blue!90!white] plot file
      {data/dh/hard/Psi_mu_005_sigma_09.dat} node[below=4pt,right] {ppp,
        $\mu=0.05$, $\sigma=0.9$, gen.~cond.~grad.};
      
      \draw[green!80!white] plot file
      {data/dh/soft/Psi_mu_005_sigma_09.dat} node[above=4pt,right] {ppp, $\mu=0.05$,
        $\sigma=0.9$, soft thresholding};
      
      \draw[red!70!white] plot file {data/dh/hard/Psi_iter_thresh.dat}
      node[above=3pt,right] {hard thresholding};
      
      \draw[black!60!white] plot file {data/dh/soft/Psi_iter_thresh.dat}
      node[below=4pt,right] {soft thresholding};
    \end{scope}
  \end{tikzpicture}
  \caption{Decay of the functional value for the problem of digital
    holoraphy example and different methods. The parameters are
    similar toFigure~\ref{fig:rec_dh}.}
  \label{fig:func_val_dh}
\end{figure}
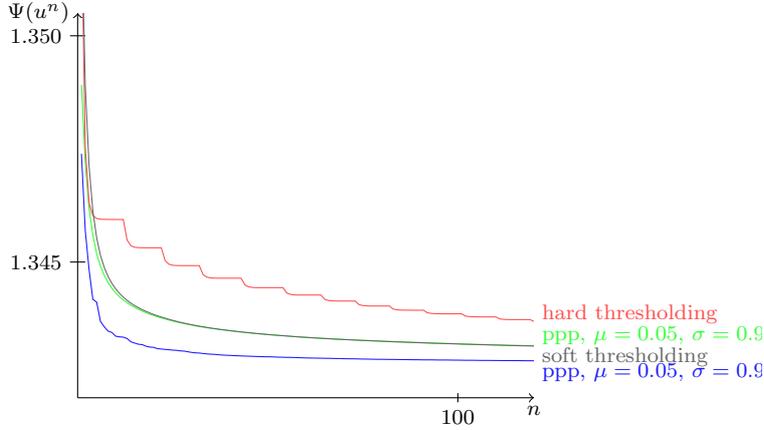

\subsection{Deblurring and the influence of $\mu$ and $\sigma$ }
\label{sec:debl-infl-mu}

The parameters $\mu$ and $\sigma$ influence the behavior of the
algorithm. Basicallly, the parameter $\mu$ influences how much the
subproblems are regularized and hence, for larger values of $\mu$ the
subproblems are solved faster, see, e.g.~the rate of convergence in
Theorem~\ref{thm:conv_damp_iter_soft}. The parameter $\sigma$ tunes
the desired accuracy for the subproblems, i.e.~the smaller $\sigma$
is, the more accurate is the solution of the subproblem and hence, the
subproblems are solved slower.  As a rule of thumb one may remember
that it does not seem necessary to choose $\sigma$ too small, i.e.~it
is enough to solve the subproblems only roughly. Typically
$\sigma=0.9$ is a good choice. On the other hand, not too much
regularization is necessary to terminate the inner loop quickly,
i.e.~small values of $\mu$ give good results. A typical value for may
be $\mu=0.05$. Moreover, smaller $\mu$ often lead to faster decay of
the functional value in the experiments.

To illustrate this behaovior we performed a one dimensional
deconvolution experiment.  We considered a discretized linear blurring
operator $A$ which consisted of the circular convolution with the
kernel $\kappa(x) = 1/(1+x^2/5^2)$ and combinded this with a synthesis
operator $B$ associated with simple hat-functions and hence,
considered the operator $K = AB$. We generated data which just
consists of a few spikes and hence, has a sparse representation in hat
functions.  We ran the projection proximal-point algorithm
(Algorithm~\ref{alg:ppp_l1}) with different values for $\mu$ and
$\sigma$ and with both soft-thresholding
(Algorithm~\ref{alg:dampitersoft}) and the generalized conditional
gradient method (Algorithm~\ref{alg:gencondgrad}) in the inner loop.
The decay of the objective value as shown in
Figure~\ref{fig:smooth_soft_mu_sigma}
and~\ref{fig:smooth_hard_mu_sigma} is typical.

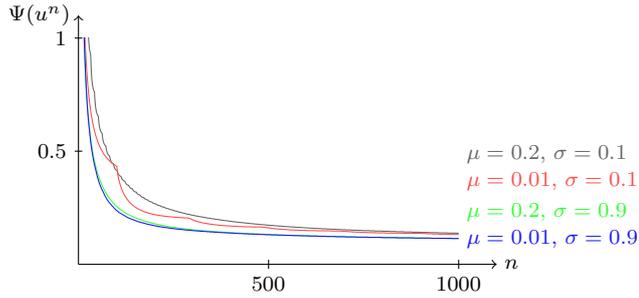
\begin{figure}
  \centering
  \begin{tikzpicture}[xscale=.05,yscale=.3]
    \footnotesize
    \draw[->] (0,0) -- (110,0) node[right] {$n$};
    \foreach \x in {50,100}
    \draw (\x,0.2) -- (\x,-0.2) node[below] {\footnotesize$\x0$};
    \draw[->] (0,0) -- (0,11) node[left] {$\Psi(u^n)$};
    \foreach \y/\ytrue in {5/0.5,10/1}
    \draw (1,\y) -- (-1,\y) node[left] {\footnotesize$\ytrue$};  
    \begin{scope}
      \clip (-1,-1) rectangle (180,10);
      \draw[black!60!white] plot file
      {data/smoothing/soft/Psi_mu_020_sigma_01.dat} node[above=30pt,right]
      {$\mu=0.2$, $\sigma=0.1$};
      \draw[red!70!white] plot file
      {data/smoothing/soft/Psi_mu_001_sigma_01.dat} node[above=20pt,right]
      {$\mu=0.01$, $\sigma=0.1$};
      \draw[green!80!white] plot file
      {data/smoothing/soft/Psi_mu_020_sigma_09.dat} node[above=10pt,right]
      {$\mu=0.2$, $\sigma=0.9$};
      \draw[blue!90!white] plot file
      {data/smoothing/soft/Psi_mu_001_sigma_09.dat} node[above=0pt,right]
      {$\mu=0.01$, $\sigma=0.9$};
    \end{scope}
  \end{tikzpicture}  
  \caption{Decay of the functional value for the projection
    proximal-point algorithtm with soft-thresholding in the inner loop
    for the deconvolution problem and different values of
    $\sigma$ and $\mu$.}
  \label{fig:smooth_soft_mu_sigma}
\end{figure}

\begin{figure}
  \centering
  \begin{tikzpicture}[xscale=.05,yscale=.5]
    \footnotesize
    \draw[->] (0,0) -- (110,0) node[right] {$n$};
    \foreach \x in {50,100}
    \draw (\x,0.2) -- (\x,-0.2) node[below] {\footnotesize$\x0$};
    \draw[->] (0,0) -- (0,6) node[left] {$\Psi(u^n)$};
    \foreach \y/\ytrue in {2.5/0.25,5/0.5}
    \draw (1,\y) -- (-1,\y) node[left] {\footnotesize$\ytrue$};  
    \begin{scope}
      \clip (-1,-1) rectangle (180,5.5);
      \draw[black!60!white] plot file
      {data/smoothing/hard/Psi_mu_020_sigma_01.dat} node[above=30pt,right]
      {$\mu=0.2$, $\sigma=0.1$};
      \draw[red!70!white] plot file
      {data/smoothing/hard/Psi_mu_001_sigma_01.dat} node[above=20pt,right]
      {$\mu=0.01$, $\sigma=0.1$};
      \draw[green!80!white] plot file
      {data/smoothing/hard/Psi_mu_020_sigma_09.dat} node[above=10pt,right]
      {$\mu=0.2$, $\sigma=0.9$};
      \draw[blue!90!white] plot file
      {data/smoothing/hard/Psi_mu_001_sigma_09.dat} node[above=0pt,right]
      {$\mu=0.01$, $\sigma=0.9$};
    \end{scope}
  \end{tikzpicture}
  \caption{Decay of the functional value for the projection
    proximal-point algorithm with the generalized conditional
    gradient method in the inner loop for the deconvolution
    problem and different values of $\sigma$ and $\mu$.}
  \label{fig:smooth_hard_mu_sigma}
\end{figure}
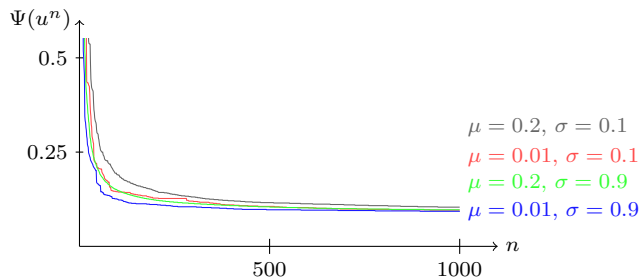

To get an impression, how the values of the parameter $\mu$ and
$\sigma$ influence the termination of the inner loop due to
conditons~(\ref{eq:cond_v_in_T_y})--(\ref{eq:cond_epsilon_less}), we
recorded the numbers of inner and outer iterations for different
values of $\mu$ and $\sigma$ in Table~\ref{tab:smooth_soft_mu_sigma}
for soft thresholding and in Table~\ref{tab:smooth_hard_mu_sigma} for
the generalized conditional gradient method. In all cases we ran the
algorithm until 350 total iterations have been done. (Note that this
sometimes causes that the last inner iteration did not terminate, see
e.g.~Table~\ref{tab:smooth_soft_mu_sigma} bottom left). First one
notes that smaller values of $\sigma$ cause the inner iteration to
take longer to terminate and the same holds true for smaller values of
$\mu$. Moreover, we see that smaller values of $\sigma$ (i.e.~solving
the subproblems more precise) does not improve the decay of the
functional value. Hence, we may conclude, that a larger value of
$\sigma$ is preferable in general. The value of $\mu$ behaves
differently for soft-thresholding and the generalized conditional
gradient method. While for the first, smaller values of $\mu$ lead to
better results concerning the functional value, for the latter the
converse holds true. Finally, the generalized conditional gradient
method gives in general smaller functional values.

\begin{table}
  \centering
  \small
  \begin{tabular}{cc}
    $\mu=0.2$, $\sigma=0.1$
    &
    $\mu=0.2$, $\sigma=0.9$\\
    \input{data/table_smoothing_soft_mu_020_sigma_01.tex}
    &
    \input{data/table_smoothing_soft_mu_020_sigma_09.tex}\\
    & \\
    $\mu=0.01$, $\sigma=0.1$
    &
    $\mu=0.01$, $\sigma=0.9$\\
    \input{data/table_smoothing_soft_mu_001_sigma_01.tex}
    &
    \input{data/table_smoothing_soft_mu_001_sigma_09.tex}\\
  \end{tabular}
  \caption{Output of the projection proximal-point algorithm with soft thresholding in the inner loop. The problem under consideration in the deblurring problem.}
  \label{tab:smooth_soft_mu_sigma}
\end{table}

\begin{table}
  \centering
  \small
  \begin{tabular}{cc}
    $\mu=0.2$, $\sigma=0.1$
    &
    $\mu=0.2$, $\sigma=0.9$\\
    \input{data/table_smoothing_hard_mu_020_sigma_01.tex}
    &
    \input{data/table_smoothing_hard_mu_020_sigma_09.tex}\\
    & \\
    $\mu=0.01$, $\sigma=0.1$
    &
    $\mu=0.01$, $\sigma=0.9$\\
    \input{data/table_smoothing_hard_mu_001_sigma_01.tex}
    &
    \input{data/table_smoothing_hard_mu_001_sigma_09.tex}\\
  \end{tabular}
  \caption{Output of the projection proximal-point algorithm with the generalized conditional gradient method in the inner loop. The problem under consideration in the deblurring problem.}
  \label{tab:smooth_hard_mu_sigma}
\end{table}

\section{Conclusion}
\label{sec:conclusion}

The projection proximal point algorithm as proposed
in~\cite{solodov1999hybridproximalpoint} is applicable for the problem
of $\ell^1$-minimization. It could be shown that under the FBI
assumption the algorithm even converges linearly. Both the iterated
soft-thresholding and the generalized conditional gradient method are
applicable methods to solve the regularized subproblems. For the
iterated soft-thresholding the projection proximal-point algorithm
does not lead to a significant improvement. On the theoretical side,
the iterated soft-thresholding itself converges
linearly~\cite{bredies2008itersoftconvlinear,hale2008fixedpointcontinuation}
and on the practical side, both methods behave comparable (see
Figure~\ref{fig:func_val_dh}). However, for the generalized
conditional gradient method the situation is different. For the
algorithm itself as stated in~\cite{bredies2008harditer} the distance
to the minimizer only converges like $\bigO(n^{-1/2})$ but the
projection proximal-point extension gives linear convergence.
Moreover, also the practical behavior is better, see
Figure~\ref{fig:func_val_dh}. In total, it seems that the generalized
conditional gradient method combines well with the projection proximal
point idea.

However, the aim of this paper was not to develop the fastest
available method for $\ell^1$-minimization but to provide an algorithm
from a different class of methods. Hence, it may be expected that
after considerable fine tuning (parameter choice rules, more efficient
solvers for the subproblems, backtracking,\dots) the projection
proximal-point algorithm will become comparable to state-of-the-art
methods.

Finally we remark, that Algorithm~\ref{alg:ppp_l1} does not rest upon
linearity of the operator and may also be applied to non-linear
operators. Also for the solution of the subproblems either iterated
soft-thresholding~\cite{bonesky2007gencondgradnonlin} or the
generalized condition gradient method may be used.

\bibliography{test}

\bibliographystyle{plain}

\end{document}

%% file: data/table_smoothing_soft_mu_020_sigma_01.tex
\begin{tabular}{@{}ccc@{}}
 \toprule
 \# out iter & \# in inter & $\Psi(u^n)$ \\
\midrule
 \verb|  1| & \verb|  6| & \verb|5.03e-02|  \\
 \verb|  2| & \verb|  8| & \verb|1.91e-02|  \\
 \verb|  3| & \verb|  9| & \verb|1.21e-02|  \\
%  \verb|  4| & \verb|  9| & \verb|9.27e-03|  \\
%  \verb|  5| & \verb| 10| & \verb|7.62e-03|  \\
%  \verb|  6| & \verb| 10| & \verb|6.55e-03|  \\
%  \verb|  7| & \verb| 10| & \verb|5.79e-03|  \\
%  \verb|  8| & \verb| 10| & \verb|5.23e-03|  \\
%  \verb|  9| & \verb| 10| & \verb|4.78e-03|  \\
%  \verb| 10| & \verb| 10| & \verb|4.43e-03|  \\
%  \verb| 11| & \verb| 10| & \verb|4.13e-03|  \\
%  \verb| 12| & \verb| 10| & \verb|3.88e-03|  \\
%  \verb| 13| & \verb| 10| & \verb|3.67e-03|  \\
%  \verb| 14| & \verb| 10| & \verb|3.49e-03|  \\
%  \verb| 15| & \verb| 10| & \verb|3.33e-03|  \\
%  \verb| 16| & \verb| 10| & \verb|3.19e-03|  \\
%  \verb| 17| & \verb| 10| & \verb|3.06e-03|  \\
%  \verb| 18| & \verb| 10| & \verb|2.95e-03|  \\
%  \verb| 19| & \verb| 10| & \verb|2.85e-03|  \\
%  \verb| 20| & \verb| 11| & \verb|2.75e-03|  \\
%  \verb| 21| & \verb| 11| & \verb|2.67e-03|  \\
%  \verb| 22| & \verb| 11| & \verb|2.59e-03|  \\
%  \verb| 23| & \verb| 11| & \verb|2.52e-03|  \\
%  \verb| 24| & \verb| 11| & \verb|2.46e-03|  \\
%  \verb| 25| & \verb| 11| & \verb|2.40e-03|  \\
%  \verb| 26| & \verb| 11| & \verb|2.35e-03|  \\
%  \verb| 27| & \verb| 11| & \verb|2.30e-03|  \\
%  \verb| 28| & \verb| 11| & \verb|2.25e-03|  \\
%  \verb| 29| & \verb| 11| & \verb|2.21e-03|  \\
%  \verb| 30| & \verb| 11| & \verb|2.17e-03|  \\
 \vdots      & \vdots     & \vdots          \\
 \verb| 31| & \verb| 11| & \verb|2.13e-03|  \\
 \verb| 32| & \verb| 11| & \verb|2.10e-03|  \\
 \verb| 33| & \verb| 11| & \verb|2.06e-03|  \\
 \verb| 34| & \verb| 11| & \verb|2.03e-03|  \\
\bottomrule
\end{tabular}

%% file: data/table_smoothing_soft_mu_020_sigma_09.tex
\begin{tabular}{@{}ccc@{}}
 \toprule
 \# out iter & \# in inter & $\Psi(u^n)$ \\
\midrule
 \verb|  1| & \verb|  1| & \verb|8.88e-02|  \\
 \verb|  2| & \verb|  1| & \verb|5.20e-02|  \\
 \verb|  3| & \verb|  1| & \verb|3.68e-02|  \\
 \vdots      & \vdots     & \vdots           \\
 \verb| 347| & \verb|  1| & \verb|1.46e-03|  \\
 \verb| 348| & \verb|  1| & \verb|1.46e-03|  \\
 \verb| 349| & \verb|  1| & \verb|1.46e-03|  \\
 \verb| 350| & \verb|  1| & \verb|1.46e-03|  \\
\bottomrule
\end{tabular}

%% file: data/table_smoothing_soft_mu_001_sigma_01.tex
\begin{tabular}{@{}ccc@{}}
 \toprule
 \# out iter & \# in inter & $\Psi(u^n)$ \\
\midrule
 \verb|  1| & \verb| 101| & \verb|5.06e-03|  \\
 \verb|  2| & \verb| 190| & \verb|2.06e-03|  \\
\bottomrule
\end{tabular}

%% file: data/table_smoothing_soft_mu_001_sigma_09.tex
\begin{tabular}{@{}ccc@{}}
 \toprule
 \# out iter & \# in inter & $\Psi(u^n)$ \\
\midrule
 \verb|  1| & \verb|  6| & \verb|2.08e-01|  \\
 \verb|  2| & \verb| 11| & \verb|1.49e-01|  \\
 \verb|  3| & \verb| 12| & \verb|9.76e-02|  \\
%  \verb|  4| & \verb| 16| & \verb|5.75e-02|  \\
%  \verb|  5| & \verb| 21| & \verb|3.13e-02|  \\
%  \verb|  6| & \verb| 25| & \verb|1.71e-02|  \\
%  \verb|  7| & \verb| 27| & \verb|9.69e-03|  \\
%  \verb|  8| & \verb| 28| & \verb|6.05e-03|  \\
%  \verb|  9| & \verb| 26| & \verb|4.24e-03|  \\
%  \verb| 10| & \verb| 23| & \verb|3.27e-03|  \\
%  \verb| 11| & \verb| 21| & \verb|2.68e-03|  \\
%  \verb| 12| & \verb| 19| & \verb|2.32e-03|  \\
%  \verb| 13| & \verb| 18| & \verb|2.07e-03|  \\
%  \verb| 14| & \verb| 16| & \verb|1.91e-03|  \\
%  \verb| 15| & \verb| 15| & \verb|1.79e-03|  \\
 \vdots      & \vdots     & \vdots          \\
 \verb| 16| & \verb| 14| & \verb|1.70e-03|  \\
 \verb| 17| & \verb| 14| & \verb|1.63e-03|  \\
 \verb| 18| & \verb| 13| & \verb|1.58e-03|  \\
 \verb| 19| & \verb| 13| & \verb|1.54e-03|  \\
\bottomrule
\end{tabular}

%% file: data/table_smoothing_hard_mu_020_sigma_01.tex
\begin{tabular}{@{}ccc@{}}
 \toprule
 \# out iter & \# in inter & $\Psi(u^n)$ \\
\midrule
 \verb|  1| & \verb|  5| & \verb|4.94e-02|  \\
 \verb|  2| & \verb|  5| & \verb|1.92e-02|  \\
 \verb|  3| & \verb|  7| & \verb|1.18e-02|  \\
 \vdots      & \vdots     & \vdots           \\
 \verb| 117| & \verb|  2| & \verb|1.27e-03|  \\
 \verb| 118| & \verb|  6| & \verb|1.27e-03|  \\
 \verb| 119| & \verb|  1| & \verb|1.26e-03|  \\
 \verb| 120| & \verb|  1| & \verb|1.26e-03|  \\
\bottomrule
\end{tabular}

%% file: data/table_smoothing_hard_mu_020_sigma_09.tex
\begin{tabular}{@{}ccc@{}}
 \toprule
 \# out iter & \# in inter & $\Psi(u^n)$ \\
\midrule
 \verb|  1| & \verb|  1| & \verb|1.54e-01|  \\
 \verb|  2| & \verb|  1| & \verb|7.78e-02|  \\
 \verb|  3| & \verb|  1| & \verb|4.36e-02|  \\
 \vdots      & \vdots     & \vdots           \\
 \verb| 347| & \verb|  1| & \verb|1.10e-03|  \\
 \verb| 348| & \verb|  1| & \verb|1.10e-03|  \\
 \verb| 349| & \verb|  1| & \verb|1.10e-03|  \\
 \verb| 350| & \verb|  1| & \verb|1.10e-03|  \\
\bottomrule
\end{tabular}

%% file: data/table_smoothing_hard_mu_001_sigma_01.tex
\begin{tabular}{@{}ccc@{}}
 \toprule
 \# out iter & \# in inter & $\Psi(u^n)$ \\
\midrule
 \verb|  1| & \verb| 27| & \verb|4.59e-03|  \\
 \verb|  2| & \verb| 32| & \verb|1.97e-03|  \\
 \verb|  3| & \verb| 85| & \verb|1.60e-03|  \\
 \verb|  4| & \verb| 83| & \verb|1.44e-03|  \\
 \verb|  5| & \verb|  1| & \verb|1.33e-03|  \\
 \verb|  6| & \verb| 12| & \verb|1.27e-03|  \\
 \verb|  7| & \verb| 97| & \verb|1.22e-03|  \\
 \verb|  8| & \verb|  1| & \verb|1.19e-03|  \\
\bottomrule
\end{tabular}

%% file: data/table_smoothing_hard_mu_001_sigma_09.tex
\begin{tabular}{@{}ccc@{}}
 \toprule
 \# out iter & \# in inter & $\Psi(u^n)$ \\
\midrule
 \verb|  1| & \verb|  5| & \verb|1.84e-01|  \\
 \verb|  2| & \verb|  4| & \verb|5.94e-02|  \\
 \verb|  3| & \verb|  4| & \verb|2.51e-02|  \\
 \vdots      & \vdots     & \vdots           \\
 \verb| 56| & \verb|  3| & \verb|1.03e-03|  \\
 \verb| 57| & \verb|  6| & \verb|1.05e-03|  \\
 \verb| 58| & \verb| 44| & \verb|1.02e-03|  \\
 \verb| 59| & \verb|  1| & \verb|1.01e-03|  \\
\bottomrule
\end{tabular}